\documentclass[12pt]{amsart}
\usepackage{amssymb, amscd, amsmath, amsthm, latexsym, enumerate}
\newtheorem{theorem}{Theorem}
\newtheorem{lemma}[theorem]{Lemma}

\begin{document}

\title{An aspherical 5-manifold with perfect fundamental group}

\author{J.A. Hillman }
\address{School of Mathematics and Statistics, University of Sydney,
 \newline
 NSW 2006,      Australia}
 
\email{jonathan.hillman@sydney.edu.au}



\begin{abstract}
We construct aspherical closed orientable 5-manifolds 
with perfect fundamental group.
This completes earlier work on $PD_n$-groups with pro-$p$ completion a 
pro-$p$ Poincar\'e duality group of dimension $\leq{n-2}$.
We also consider the question of whether there are any examples with
``dimension drop" 1.
\end{abstract}

\maketitle
The paper \cite{HKL13} considers the phenomenon of dimension 
drop on pro-$p$ completion for orientable Poincar\'e duality groups. 
Products of aspherical homology spheres, copies of $S^1$ and
copies of the 3-manifold $M(K)$ obtained by 0-framed surgery 
on a nontrivial prime knot give examples of aspherical
closed orientable $n$-manifolds $N$ such that the pro-$p$ 
completion of $\pi_1(N)$ is a pro-$p$ Poincar\'e duality group 
of dimension $r$, 
for all $n\geq{r+2}$ and $r\geq0$, except when $n=5$ and $r=0$.
(This gap reflects the fact that 5 is not in the additive semigroup 
generated by 3 and 4,
dimensions in which aspherical homology spheres are known.)

We fill this gap in Theorem 2 below.
Modifying the construction of \cite{RT05} gives an aspherical closed 4-manifold 
with perfect fundamental group and non-trivial second cohomology.
The total space of a suitable $S^1$-bundle over this 4-manifold 
has the required properties.
We then apply Theorem 2 to refine the final result of \cite{HKL13}.
No such examples with dimension drop ${n-r}=1$ are known as yet.
The lowest dimension in which there might be such examples is $n=4$,
and we consider this case in the final section.
If they exist, products with copies of $S^1$ would give examples in all higher dimensions.

In this paper all manifolds and $PD_n$-groups shall be orientable.
If $G$ is a group then $G'$ and $G_{[n]}$ shall denote the commutator subgroup
and the $n$th term of the lower central series, respectively.
Let $G(\mathbb{Z})=\cap_{\lambda\in{Hom}(G,\mathbb{Z})}\mathrm{Ker}(\lambda)$.
Then $G/G(\mathbb{Z})$ is the maximal torsion-free abelian quotient of $G$.
If $p\geq2$ let $X^p(G)$ be the verbal subgroup generated by all $p$th powers in $G$.

\section{an aspherical 5-manifold with $\pi_1$ perfect}

Let $X$ be a compact $4$-manifold whose boundary components 
are diffeomorphic to the 3-torus $T^3$.
A {\it Dehn filling\/} of a component $Y$ of $\partial{X}$ is the adjunction of $T^2\times{D^2}$ 
to $X$ via a diffeomorphism $\partial(T^2\times{D^2})\cong{Y}$.
If the interior of $X$ has a complete hyperbolic metric
then ``most" systems of Dehn fillings on some or all of the boundary components 
give manifolds which admit metrics of non-positive curvature, 
and the fundamental groups of the cores of the solid tori $T^2\times{D^2}$ 
map injectively to the fundamental group of the filling of $X$,
by the Gromov-Thurston $2\pi$-Theorem.
(Here ``most" means ``excluding finitely many fillings of each boundary component".
See \cite{An06}.)

\begin{theorem}
There are aspherical closed $4$-manifolds $M$ with perfect fundamental group
and $H^2(M;\mathbb{Z})\not=0$.
\end{theorem}

\begin{proof}
Let $M=S^4\setminus5T^2$ be the complete hyperbolic 4-manifold with 
finite volume and five cusps considered in \cite{Iv04} and \cite{RT05}, 
and let $\overline{M}$ be a compact core, with interior diffeomorphic to $M$.
Then $H_1(\overline{M};\mathbb{Z})\cong\mathbb{Z}^5$, $\chi(\overline{M})=2$ 
and the boundary components of $\overline{M}$ are all diffeomorphic to the 3-torus $T^3$.
There are  infinitely many quintuples of Dehn fillings of the components of $\partial\overline{M}$ 
such that the resulting closed 4-manifold is an aspherical homology 4-sphere \cite{RT05}.
Let $\widehat{M}$ be one such closed 4-manifold, and let $N\subset\widehat{M}$ 
be the compact 4-manifold obtained by leaving one boundary component of $X$ unfilled.
We may assume that the interior of $N$ has a non-positively curved metric,
and so $N$ is aspherical.
The Mayer-Vietoris sequence for $M=N\cup{T^2\times{D^2}}$  gives an isomorphism
\[
H_1(T^3;\mathbb{Z})\cong{H_1(N;\mathbb{Z})}\oplus{H_1(T^2;\mathbb{Z})}.
\]
Let $\{x,y,z\}$ be a basis for $H_1(T^3;\mathbb{Z})$ compatible with this splitting.
Thus $x$ represents a generator of $H_1(N;\mathbb{Z})$ and maps to 0 in the second summand,
while $\{y,z\}$ has image 0 in $H_1(N;\mathbb{Z})$ but generates the second summand.
Since the subgroup generated by $\{y,z\}$ maps injectively to $\pi_1(\widehat{M)}$ \cite{An06},
the inclusion of $\partial{N}$  into $N$ is $\pi_1$-injective.
Let $\phi$ be the automorphism of $\partial{N}=T^3$ which swaps the generators $x$ and $y$,
and let $P=N\cup_\phi{N}$.
Then $P$ is aspherical and $\chi(P)=2\chi(N)=4$.
A Mayer-Vietoris calculation gives $H_1(P;\mathbb{Z})=0$, 
and so $\pi=\pi_1(P)$ is perfect and $H^2(P;\mathbb{Z})\cong\mathbb{Z}^2$.
\end{proof}

Are there compact complex surfaces with perfect fundamental group
and which are uniformized by the unit ball or the bidisc?
(Such manifolds are aspherical, and have middle cohomology of rank $>0$,
since they are K\"ahler.)

\begin{theorem}
There are aspherical closed $5$-manifolds with perfect fundamental group.
\end{theorem}

\begin{proof}
Let $P$ be an aspherical closed 4-manifold with $\pi_1(P)$ perfect and $H^2(P;\mathbb{Z})\not=0$, as in Theorem 1.
Let $e$ generate a direct summand of $H^2(\pi;\mathbb{Z})=H^2(P;\mathbb{Z})$, 
and let $E$ be the total space of the $S^1$-bundle over $P$ with Euler class $e$.
Then $E$ is an aspherical 5-manifold, 
and $G=\pi_1(E)$ is the central extension of $\pi_1(P)$ by $\mathbb{Z}$ 
corresponding to $e\in {H^2(\pi_1(P);\mathbb{Z})}$. 
The Gysin sequence for the bundle (with coefficients in $\mathbb{F}_p$) has a subsequence
\[
0\to{H^1}(E;\mathbb{F}_p)\to{H^0(P;\mathbb{F}_p)}\to{H^2(P;\mathbb{F}_p)}\to
{H^2(E;\mathbb{F}_p)}\to\dots
\] 
in which the {\it mod}-$p$ reduction of $e$ generates the image of $H^0(P;\mathbb{F}_p)$.
Since $e$ is indivisible this image is nonzero, for all primes $p$.
Therefore $H^1(G;\mathbb{F}_p)=H^1(E;\mathbb{F}_p)=0$, for all $p$, and so $G$ is perfect.
\end{proof}

(From the algebraic point of view, 
$G$ is a quotient of the universal central extension of $\pi_1(P)$,
which is perfect.)

Is there an aspherical 5-dimensional homology sphere?
If there is an aspherical 4-manifold $X$ with the integral homology of 
$\mathbb{CP}^2$ then the total space of the $S^1$-bundle with Euler class
a generator of $H^2(X;\mathbb{Z})$ would be such an example.
The fake projective planes of \cite{PY07} are asperical and
have the rational homology of $\mathbb{CP}^2$, 
but they have nonzero first homology.
The total spaces of the $S^1$-bundles over such fake projective planes
and with Euler class of infinite order are aspherical
rational homology 5-spheres.

\section{pro-$p$ completions and dimension drop $\geq2$}

The simplest construction of examples of orientable $PD_n$ groups
with pro-$p$ completion a pro-$p$ Poincar\'e duality group 
of lower dimension uses the fact that finite $p$-groups are nilpotent.
Thus if $G$ is a group with $G'=[G,G']$ the abelianization homomorphism 
induces isomorphisms on pro-$p$ completions, for all primes $p$.
Hence products of perfect groups with free abelian groups have 
pro-$p$ completion a free abelian pro-$p$ group.
In this section we shall use Theorem 2 to remove a minor constraint 
on the final result of \cite{HKL13}, which excluded a family of such examples.

\begin{theorem}
For each $r\geq0$ and  $n\geq\max\{r+2,3\}$ there is an aspherical closed $n$-manifold
with fundamental group $\pi$ such that $\pi/\pi'\cong\mathbb{Z}^r$ and $\pi'=\pi''$.
\end{theorem}

\begin{proof}
Let $\Sigma$ be an aspherical homology 3-sphere 
(such as the Brieskorn 3-manifold $\Sigma(2,3,7)$)
and let $P$ and $E$ be as in Theorems 1 and 2. 
Taking suitable products of copies of $\Sigma$, $P$, $E$  and $S^1$
with each other realizes all the possibilities with $n\geq{r+3}$,
for all $r\geq0$.

Let $M=M(K)$ be the 3-manifold obtained by $0$-framed surgery 
on a nontrivial prime knot $K$ with Alexander polynomial $\Delta(K)=1$ 
(such as the Kinoshita-Terasaka  knot $11_{n42}$).
Then $M$ is aspherical, since $K$ is nontrivial \cite{Ga87}, and
if $\mu=\pi_1(M)$ then $\mu/\mu'\cong\mathbb{Z}$ and $\mu'$ is perfect,
since $\Delta(K)=1$.
Hence products $M\times(S^1)^{r-1}$ give examples with $n=r+2$, for all $r\geq1$.
\end{proof}

In particular, the dimension hypotheses in Theorem 6.3 
of \cite{HKL13} may be simplified,
so that it now asserts:

{\it
Let  $m\geq3$ and $r\geq0$.
Then there is an aspherical closed $(m+r)$-manifold $M$ with fundamental group 
$G=K\times\mathbb{Z}^r$, where $K=K'$.
If $m\not=4$  we may assume that $\chi(M)=0$, and if $r>0$ this must be so.}\\
This is best possible, as no $PD_1$- or $PD_2$-group is perfect,
and no perfect $PD_4$-group $H$ has $\chi(H)=0$.

\section{dimension drop $\leq1$?}

Can we extend Theorem 3 to give examples of $PD_n$-groups $\pi$ with
$\pi/\pi'\cong\mathbb{Z}^{n-1}$ and $\pi_{[2]}=\pi_{[3]}$, 
realizing dimension drop 1 on all pro-$p$ completions?
In this section we shall weaken some of these conditions,
by considering maps to $PD_{n-1}$ groups other than $\mathbb{Z}^{n-1}$ 
and requiring only that pro-$p$ completion be well-behaved 
for {\it some\/} primes $p$.

There are clearly no such examples with $n=2$,
since the only pro-$p$ Poincar\'e duality group of dimension 1 is $\widehat{\mathbb{Z}}_p$.
The next lemma rules out a direct analogue of Theorem 3 with $n=3$ and $r=n-1$.

\begin{lemma}
Let $\pi$ be a $PD_3$-group such that $\pi/\pi'\cong\mathbb{Z}^r$ and $\pi'=\pi''$. 
Then $r=0,1$ or $3$.
\end{lemma}

\begin{proof}
We may assume that $r>1$. 
The augmentation $\pi$-module $\mathbb{Z}$ has  a finitely generated 
projective resolution $C_*$ of length 3.
Let $\Lambda=\mathbb{Z}[\pi/\pi']$,
and let $D_*=\Lambda\otimes_\pi{C_*}$ and $D^*=Hom_\Lambda(D_{3-*},\Lambda)$.
Then $H^p(D^*)\cong{H_{3-p}(D_*)}$, by Poincar\'e duality for $\pi$.
We have $H_1(D_*)=0$, since $\pi'$ is perfect, 
$H_2(D_*)\cong{H^1(D^*)}\cong {Ext}^1_\Lambda(\mathbb{Z},\Lambda)=0$,
since $r>1$, and 
$H_3(D_*)\cong{H^0(D^*)}\cong {Hom}_\Lambda(\mathbb{Z},\Lambda)=0$,
since $r>0$.
Therefore $D_*$ is a finitely generated free resolution of the augmentation
$\pi/\pi'$-module.
Since $D_*$ has length 3 and $H^3(D^*)\cong{H_0(D_*)}=\mathbb{Z}$, 
we must have $r=3$.
\end{proof}

The values $r=0,1$ and 3 may be realized by the fundamental groups
of $\Sigma(2,3,7)$), $M(11_{n42})$ and the 3-torus $(S^1)^3$, respectively.

Pro-$p$ Poincar\'e duality groups of dimension 2 are also well understood.
This class (of so-called {\it Demu\v skin groups}) includes all pro-$p$ completions of $PD_2$-groups, but is somewhat larger.

\begin{theorem}
The pro-$p$ completion of a $PD_3$-group $G$ is not a Demu\v skin group,
for any odd prime $p$.
\end{theorem}

\begin{proof}
If $G$ is a discrete group and $p$ is an odd prime then 
the kernel of cup product from $\wedge^2H^1(G;\mathbb{F}_p)$ 
to $H^2(G;\mathbb{F}_p)$ is isomorphic to
$G_{[2]}X^p(G)/G_{[3]}X^p(G)$ \cite{Hi87}, 
and a similar result holds for pro-$p$ groups \cite{Wu}.
Hence the dimension of the image of this cup product is determined 
by the $p$-lower central series.
In particular,
if the pro-$p$ completion of $G$ is a Demu\v skin group
the image of $\wedge^2H^1(G;\mathbb{F}_p)$ in $H^2(G;\mathbb{F}_p)$ is 1-dimensional.
If $G$ is a $PD_3$-group this is impossible, by the non-singularity of Poincar\'e duality.
\end{proof}

This argument can be modified to apply for $p=2$ also.

In higher dimensions  the most convenient candidates for quotients are 
torsion-free nilpotent groups.
A finitely generated nilpotent group $\nu$ of Hirsch length $h$ 
has a maximal finite normal subgroup $T(\nu)$, with quotient a $PD_h$-group.
Moreover, $\nu/T(\nu)$ has  nilpotency class $<h$,
and is residually a finite $p$-group for all $p$,
by Theorem 4 of Chapter 1 of \cite{Se}.
Thus the pro-$p$ completion of $\nu$ is a pro-$p$ Poincar\'e duality group 
for all $p$ prime to the order of $T(\nu)$.

We shall focus on the first undecided case, $n=4$.
If $G_{[k]}/G_{[k+1]}$ is finite, of exponent $e$, say,
then so are all subsequent subquotients of the lower central series,
by Proposition 11 of Chapter 1 of \cite{Se}.
Thus if $G$ is a $PD_4$-group such that $G/G_{[3]}$ has Hirsch length 3
and $G_{[3]}/G_{[4]}$ is finite then, setting $\nu=G/G_{[3]}$,
the canonical projection to $\nu/T(\nu)$ induces isomorphisms on pro-$p$ completions, 
for almost all primes $p$.
Taking products of one such group with copies of $\mathbb{Z}$
would give similar examples with dimension drop 1 in all higher dimensions.

We consider first the case when the quotient $PD_3$-group is abelian.

\begin{theorem}
Let $G$ be a $PD_4$-group.
Then there is an epimorphism from $G$ to $\mathbb{Z}^3$
which induces isomorphisms 
on pro-$p$ completions, for almost all $p$,
if and only if  $\beta_1(G)=3$ and the homomorphism from $\wedge^2H^1(G;\mathbb{Z})$ 
to $H^2(G;\mathbb{Z})$ induced by cup product is injective.

If these conditions hold then $\chi(G)\geq2$ and $G(\mathbb{Z})$ is not $FP_2$.
\end{theorem}

\begin{proof}
The homomorphism from $\wedge^2H^1(G;\mathbb{Z})$ to 
$H^2(G;\mathbb{Z})$ induced by cup product is a monomorphism if and only if 
$G_{[2]}/G_{[3]}$ is finite \cite{Hi87}.

Suppose that there is such an epimorphism.
Since $G/G'$ and $G_{[2]}/G_{[3]}$ are finitely generated, it follows
easily that $\beta_1(G)=3$ and $G_{[2]}/G_{[3]}$ is finite.
Thus the conditions in the first assertion are necessary.

If they hold then $G_{[2]}/G_{[3]}$ is finite, 
and the rational lower central series for $G$ terminates at $G(\mathbb{Z})$ \cite{Hi87}.
Therefore $G(\mathbb{Z})/G_{[3]}$ is the torsion subgroup of $G/G_{[3]}$.
Thus if $p$ is prime to the order of $G(\mathbb{Z})/G_{[3]}$ then
the canonical epimorphism from $G$ to $G/G(\mathbb{Z})\cong\mathbb{Z}^3$ 
induces an isomorphism of pro-$p$ completions.

The image of cup product from $\wedge^2H^1(G;\mathbb{Z})$ 
to $H^2(G;\mathbb{Z})$  has rank 3, and must be self-annihilating, 
since $\wedge^4(\mathbb{Z}^3)=0$.
Hence $\beta_2(G)\geq6$, by the non-singularity of Poincar\'e duality,
and so $\chi(G)\geq2$.

If $G(\mathbb{Z})$ were $FP_2$ then it would be a $PD_1$-group,
and so $FP$, by Theorem 1.19 of \cite{Hi}.
But then $\chi(G)=0$, since $\chi$ is multiplicative in exact sequences of groups of type $FP$.
\end{proof}

The conditions in this theorem are detected by de Rham cohomology,
which suggests that we should perhaps seek examples among
the fundamental groups of smooth manifolds with metrics of negative curvature.
Are there any such groups? 
Since $\chi(G)\not=0$, no such group is solvable
or a semidirect product $H\rtimes\mathbb{Z}$ with $H$ of type $FP$.

There are parallel criteria in the nilpotent case.

\begin{theorem}
Let $G$ be a $PD_4$-group.
Then there is an epimorphism from $G$ to  a nonabelian nilpotent $PD_3$-group
which induces isomorphisms on pro-$p$ completions, for almost all $p$,
if and only if  $\beta_1(G)=2$,
cup product from $\wedge^2H^1(G;\mathbb{Z})$ to $H^2(G;\mathbb{Z})$ is $0$,
and  $G_{[3]}/G_{[4]}$ is finite.

If these conditions hold then $\chi(G)\geq0$.
\end{theorem}

\begin{proof}
The conditions are clearly necessary.
Suppose that they hold.
Then $G/G(\mathbb{Z})\cong\mathbb{Z}^2$.
The homomorphism from the free group  $F(2)$ to $G$ determined by 
elements of $G$ representing a basis for this quotient induces a monomorphism 
from $F(2)/F(2)_{[3]}$ to $G/G_{[3]}$ with image of finite index,
by the cup-product condition \cite{Hi87}. 
Thus $G(\mathbb{Z})/G_{[3]}$ is nilpotent and virtually $\mathbb{Z}$.
Let $T$ be the preimage in $G$ of the torsion subgroup of $G(\mathbb{Z})/G_{[3]}$.
This is characteristic in $G$, and  $G/T$ is a non-abelian extension 
of $\mathbb{Z}^2$ by $\mathbb{Z}$. 
Hence it is a nilpotent $PD_3$-group.
 If $G_{[3]}/G_{[4]}$ is finite, then 
the quotient epimorphism to $G/T$ induces isomorphisms
on pro-$p$ completions, for almost all $p$.

If these conditions hold then the natural map from $H_2(G;\mathbb{Q})$ to 
$H_2(G/G_{[3]};\mathbb{Q})\cong\mathbb{Q}^2$ is an epimorphism,
by the 5-term exact sequence of low degree for the homology of $G$ as an
extension of $G/G_{[3]}$ by $G_{[3]}$.
Hence $\beta_2(G)\geq2$, and so $\chi(G)\geq0$.
\end{proof}

Are there any such groups?
Theorem 1.19 of \cite{Hi} again implies that $T$ cannot be $FP_2$.
If $\chi(G)=0$ then $T$ cannot even be finitely generated, 
by Corollary 6.1 of \cite{HK07} (used twice).
For otherwise $T$ would be $\mathbb{Z}$,
so $G$ would be nilpotent, and $G_{[3]}/G_{[4]}$ would be infinite.
(If $\Gamma$ is a lattice in the nilpotent Lie group $Nil^4$ 
then the first two conditions of Theorem 6 hold, and $\chi(\Gamma)=0$, 
but $\Gamma_{[3]}\cong\mathbb{Z}$ and $\Gamma_{[4]}=1$.)
Is there such a group with $T$ free of infinite rank?

The possibility of no dimension drop ($n=r$) is realized by the $n$-torus $(S^1)^n$, for any $n$.
Are there any examples in which the dimension increases on pro-$p$ completion,
i.e., with $n<r$?
It again follows from Lemma 4 that if $\pi/\pi'\cong\mathbb{Z}^r$ 
and $\pi'$ is perfect then $n\geq4$.
Moreover,  if there is such a $PD_4$-group $G$ then $\chi(G)\geq2$,
by the non-singularity of Poincar\'e duality.

\medskip
{\it Acknowledgment.}
We would like to thank Bruno Martelli for suggesting 
the use of manifolds such as the manifold $N$ of Theorem 1.

\end{document}